\documentclass{article}
\usepackage{amsmath}
\usepackage{amssymb}
\usepackage{amsthm}
\usepackage{bbm,mathtools}
\usepackage{enumitem}
\usepackage{enumitem}

\newcounter{dummy}
\newcommand\myitem[1][]{\item[#1]\refstepcounter{dummy}\def\@currentlabel{#1}}

\newcommand{\midb}{\;\middle|\;}
\newcommand{\one}{\mathbbm 1}

\def\reals{\mathbb{R}}

\def\uball{\mathbb{B}}
\def\ereals{\overline{\mathbb{R}}}

\def\comp{\raise 1pt \hbox{$\scriptstyle\circ$}}

\def\dom{\mathop{\rm dom}\nolimits}

\def\upto{{\raise 1pt \hbox{$\scriptstyle \,\nearrow\,$}}}
\def\downto{{\raise 1pt \hbox{$\scriptstyle \,\searrow\,$}}}

\def\pos{\mathop{\rm pos}}

\def\B{{\cal B}}
\def\C{{\cal C}}

\def\E{{\cal E}}
\def\F{{\cal F}}

\def\G{{\cal G}}

\def\M{{\cal M}}

\def\P{{\cal P}}

\def\U{{\cal U}}

\def\Y{{\cal Y}}

\newtheorem{theorem}{Theorem}
\newtheorem{lemma}[theorem]{Lemma}
\newtheorem{corollary}[theorem]{Corollary}

\newtheorem{example}[theorem]{Example}
\newtheorem{remark}[theorem]{Remark}
\theoremstyle{definition}

\theoremstyle{empty}

\title{Topological duals of locally convex function spaces}

\author{Teemu Pennanen\thanks{Department of Mathematics, King's College London, Strand, London, WC2R 2LS, United Kingdom, teemu.pennanen@kcl.ac.uk} \and Ari-Pekka Perkki\"o\thanks{Mathematics Institute, Ludwig-Maximilian University of Munich, Theresienstr. 39, 80333 Munich, Germany, a.perkkioe@lmu.de. Corresponding author}}

\begin{document}

\maketitle

\begin{abstract}
This paper studies topological duals of locally convex function spaces that are natural generalizations of Fr\'echet and Banach function spaces (BFS). We assume a finite reference measure but the topology is generated by an arbitrary collection of seminorms that satisfy the usual BFS axioms. The dual is identified with the direct sum of another space of random variables (K\"othe dual), a space of purely finitely additive measures and the annihilator of $L^\infty$. In the special case of rearrangement invariant spaces, the second component in the dual vanishes and we obtain various classical as well as new duality results e.g.\ on Lebesgue, Orlicz, Lorentz-Orlicz spaces and spaces of finite moments. Beyond rearrangement invariant spaces, we find topological duals of Musielak-Orlicz spaces and those associated with general convex risk measures.
\end{abstract}

\noindent\textbf{Keywords.} Banach function spaces, topological duals, finitely additive measures
\newline
\newline
\noindent\textbf{AMS subject classification codes.} 46E30, 46A20, 28A25

\section{Introduction}

Banach function spaces (BFS) provide a convenient set up for functional analysis in spaces of measurable functions. Many well known properties of e.g.\ Lebesgue spaces and Orlicz spaces extend to BFS with minor modifications; see e.g.\ \cite{lux55,zaa67, bs88, kps82}. Much of the theory focuses on {\em rearrangement invariant} (ri) spaces where the norm of a function only depends on its distribution. Such spaces are arguably the most important among BFS but they do exclude some interesting cases such as Musielak-Orlicz spaces and spaces of random variables that arise in the theory of risk measures; see Section~\ref{sec:ex} below.

This paper studies the topological duals of locally convex spaces of random variables where the topology is generated by an arbitrary collection of seminorms that satisfy the usual properties of BFS-norms; see Section~\ref{sec:main} below. Building on the classical result of Yosida and Hewitt~\cite[Section~2]{yh52} on the dual of $L^\infty$, we identify the topological dual as the direct sum of another space of random variables (K\"othe dual), a space of purely finitely additive measures and the annihilator of $L^\infty$. The last two components have a singularity property that has been found useful in the analysis of convex integral functionals by Rockafellar~\cite{roc71} in the case of $L^\infty$ and by Kozek~\cite{koz80} in the case of Orlicz spaces. In the case of $L^\infty$, the last component in the dual vanishes while in other Orlicz spaces, the second one vanishes; see \cite{rr91}. Our result thus unifies the two seemingly complementary cases.

The main result is illustrated first by simple and unified derivations of various existing duality results in Musielak-Orlicz, Marcinkiewich, Lorentz and Orlicz-Lorentz spaces. In the last case, we also obtain an expression for the dual norm which seems to be new. We then go beyond the existing BFS settings by identifying topological duals of the space of random variables with finite moments, generalized Orlicz spaces as well as spaces of random variables associated with general convex risk measures. Such spaces have attracted attention in the recent literature on insurance and financial mathematics; see e.g.\ \cite{pic13}, \cite{ls17} and \cite{kp18}. 

The last section establishes the necessity of our axioms for locally convex spaces of random variables that are in separating duality with an other one. More precisely, a complete solid decomposable space in separating duality with another solid decomposable space has a compatible topology generated by a collection of seminorms satisfying the usual BFS-properties.

The rest of the paper is organized as follows. Section~\ref{sec:yh} reviews the duality theory for $L^\infty$. Section~\ref{sec:ei} extends the notion of an integral with respect to a finitely additive measure to measurable not necessarily bounded random variables. Section~\ref{sec:main} defines a general locally convex space of random variables and gives the main result of the paper by characterizing the topological dual of a space. 
Section~\ref{sec:ex} applies the main result to characterize the topological dual in various known and new settings. Section~\ref{sec:nec} concludes by illustrating the necessity of the employed axioms.

\section{Topological dual of $L^\infty$}\label{sec:yh}

Let $(\Omega,\F,P)$ be a probability space with a $\sigma$-algebra $\F$ and a countably additive probability measure $P$. This section gives a quick review of the Banach space $L^\infty$ of equivalence classes of essentially bounded measurable functions on a probability space $(\Omega,\F,P)$. We consider $\reals^n$-valued functions and endow $L^\infty$ with the norm
\[
\|u\|_{L^\infty}:=|(\|u_1\|_{L^\infty},\dots,\|u_n\|_{L^\infty})|,
\]
where $|\cdot|$ is a norm on $\reals^n$. The dual norm on $\reals^n$ is denoted by $|\cdot|^*$.


Let $\M^1$ be the set of $P$-absolutely continuous finitely additive $\reals^n$-valued measures on $(\Omega,\F)$ and let $\M^{1s}$ be set of those $m^s\in\M^1$ which are {\em singular} (``purely finitely additive'' in the terminology of \cite{yh52}; see \cite[Theorem~1.22]{yh52}) in the sense that there is a decreasing sequence $(A^\nu)_{\nu=1}^\infty\subset\F$ with $P(A^\nu)\downto 0$ and $|m^s|^*(\Omega\setminus A^\nu)=0$. Given $m\in\M^1$, the set function $|m|^*:\F\to\reals$ is defined by
\[
|m|^*(A):=|m^+(A)+m^-(A)|^*,
\]
where $i$th components of $m^+\in\M^1$ and $m^-\in\M^1$ are the positive and negative parts, respectively, of the $i$th component $m_i$ of $m$; see \cite[Theorem~1.12]{yh52}.


Recall that the space $\E$ of $\reals^n$-valued simple random variables (i.e. piecewise constant with a finite range) is dense in $L^\infty$. Given $m\in\M^1$, the integral of a $u\in\E$ is defined by
\[
\int_\Omega udm :=\sum_{j=1}^J \alpha^j m(A^j),
\]
where $A^j\in\F$ and $\alpha^j\in\reals^n$, $j=1,\ldots,m$ are such that $u=\sum_{j=1}^m \alpha^j 1_{A^j}$
On $L^\infty$, the integral is defined as the unique norm continuous linear extension from $\E$ to $L^\infty$.

The following is from \cite[Theorem~2.3]{yh52} except that we do not assume that the underlying measure space is complete. The proof uses \cite[Theorem~20.35]{hs75} which does not rely on the completeness but identifies the dual of $L^\infty$ with the space of finitely additive measures that are absolutely continuous with respect to $P$. Combined with the results of \cite[Section~1]{yh52} on decomposition of finitely additive measures then completes the proof. The extension to spaces of $\reals^n$-valued random variables is straightforward; see \cite[Lemma~1]{sat5} for an extension to Banach space-valued random variables.

\begin{theorem}[Yosida--Hewitt]\label{thm:YH}
The topological dual $(L^\infty)^*$ of $L^\infty$ can be identified with $\M^1$ in the sense that for every $u^*\in(L^\infty)^*$ there exist unique $m\in\M^1$ such that
\[
\langle u,u^*\rangle = \int_\Omega udm,
\]
where the integral is defined componentwise. The dual norm is given by
\[
\|m\|_{L^\infty}^*=|m|^*(\Omega).
\]
Moreover, $\M^1=L^1\oplus\M^{1s}$ in the sense that for every $m\in\M^1$ there exist unique $y\in L^1$ and $m^s\in\M^{1s}$ such that
\[
\int_\Omega udm = E[u\cdot y]+\int_\Omega udm^s.
\]
We have $m^s=0$ if and only if $\langle u1_{A^\nu},u^*\rangle\to 0$ for every $u\in L^\infty$ and every decreasing $(A^\nu)_{\nu=1}^\infty\subset\F$ such that $P(A^\nu)\downto 0$.
\end{theorem}

\begin{proof}
Assume first that $n=1$. By \cite[Theorem~20.35]{hs75}, the dual of $L^\infty$ can be identified with the linear space of finitely additive $P$-absolutely continuous measures $m$ in the sense that every $u^*\in(L^\infty)^*$ can be expressed as
\[
\langle u,u^*\rangle = \int_\Omega udm
\]
and, conversely, any such integral belongs to $(L^\infty)^*$. By \cite[Theorem~1.24]{yh52}, there is a unique decomposition $m=m^a+m^s$, where $m^a$ is countably additive and $m^s$ is purely finitely additive. The construction in \cite{yh52} also shows that $m^a$ and $m^s$ are absolutely continuous with respect to $m$ and thus, absolutely continuous with respect to $P$ as well. By \cite[Theorem~1.22]{yh52}, there is a decreasing sequence $(A^\nu)_{\nu=1}^\infty\subset\F$ such that $P(A^\nu)\searrow 0$ and $m^s(\Omega\setminus A^\nu)=0$. The functional $y^s\in(L^\infty)^*$ given by
\[
\langle u,y^s\rangle := \int_\Omega udm^s
\]
then has the property in the statement. By Radon-Nikodym, there exists a $y\in L^1$ such that
\[
\langle u,u^*\rangle := E[u\cdot y] + \int_\Omega udm^s.
\]

To prove the last claim, let $u^*\in(L^\infty)^*$ and consider the representation in terms of $y\in L^1$ and $m^s\in\M^{1s}$ given by the second claim. Let $A^\nu$ be the sets in the characterization of the singularity of $m^s$. By \cite[Theorems~1.12 and 1.17]{yh52}, $m^s=m^{s+}-m^{s-}$ for nonnegative purely finitely additive $m^{s+}$ and $m^{s-}$. Given $\epsilon>0$, \cite[Theorem~1.21]{yh52} gives the existence of $A\in\F$ such that $m^{s+}(\Omega\setminus A)<\epsilon$ and $m^{s-}(A)<\epsilon$. We have
\[
\langle u1_A1_{A^\nu},u^*\rangle = E[1_A1_{A^\nu}u\cdot y] + m^s(A\cap A^\nu) \to m^s(A) > m^{s+}(\Omega)-2\epsilon.
\]
By assumption, the left side converges to zero. Since $\epsilon>0$ was arbitrary, $m^{s+}=0$. By symmetry, we must have $m^{s-}=0$ so that $m^s=0$ which means that $u^*$ is $\tau(L^\infty,L^1)$-continuous.

By \cite[Theorem~2.3]{yh52}, the dual norm of $\|\cdot\|_{L^\infty}$ is given by $\|m\|_{TV}:=m^+(\Omega)+m^-(\Omega)$. This completes the proof of the case $n=1$. The general case follows from the fact that the dual of a Cartesian product of Banach spaces is the Cartesian product of the dual spaces with the norm
\[
\|u\|_{L^\infty}^*=|(\|m_1\|_{TV},\dots,\|m_n\|_{TV})|^*,
\]
which completes the proof.
\end{proof}



\section{Extension of the integral}\label{sec:ei}

In \cite{yh52} and in Section~\ref{sec:yh}, integrals with respect to an $m\in\M^1$ were defined only for elements of $L^\infty$ as norm-continuous extensions of integrals of simple functions. Weakening the topology, it is possible to extend the definition of the integral to a larger space of measurable functions using Daniell's construction much as in \cite[Chapter~II]{bic10} who considered countably additive integrals of arbitrary (not necessarily $\F$-measurable) functions.

Another approach to integration of unbounded functions with respect to finitely additive measures is that of Dunford; see Dunford and Schwartz~\cite{ds88} or Luxemburg~\cite{lux91}. A benefit of the Daniell extension adopted here is that it gives rise to a simpler definition of integrability that is easier to verify for larger classes of measurable functions.

Given $m\in\M^1$, we define $\rho_m:L^1\to\ereals$ by
\[
\rho_m(u) := \sup_{u'\in L^\infty}\left\{\int_\Omega u'dm \midb |u'_j|\le|u_j|\ \forall j=1,\dots,n\right\}.
\]
We denote $\dom\rho_m :=\{u\in L^1 \mid \rho_m(u)<\infty\}$.
\begin{lemma}\label{lem:ext}
The function $\rho_m$ is a seminorm on $\dom\rho_m$ and
\[
|\int_\Omega udm| \le \rho_m(u)
\]
for all $u\in L^\infty$. For every $u\in\dom\rho_m$ and $\epsilon>0$, there exists a $u'\in L^\infty$ such that $\rho_m(u-u')\le\epsilon$.
\end{lemma}

\begin{proof}
We have
\[
\rho_m (u) :=\sum_{j=1}^n \rho_{m_j}(u_j)
\]
where 
\[
\rho_{m_j}(u_j) = \sup_{u'\in L^\infty(\reals)}\{ \int_\Omega u'dm_j \mid |u'|\le |u_j|\}.
\]
Thus we may assume that $n=1$ and the claims from Theorem~\ref{thm:ext1d} in the appendix.
\end{proof}

By Lemma~\ref{lem:ext}, the integral is $\rho_m$-continuous on $L^\infty$ and $L^\infty$ is $\rho_m$-dense in $\dom \rho_m$. Thus the integral has a unique $\rho_m$-continuous linear extension to $\dom\rho_m$. We call the extension the {\em $m$-integral} of $u$ and denote it by
\[
\int_\Omega udm.
\]
The elements of $\dom\rho_m$ will be said to be {\em $m$-integrable}. If $m$ is countably additive, then, e.g., by the interchange rule~\cite[Theorem~14.60]{rw98},
\[
\rho_m(u) = \sum_{j=1}^n\int_\Omega|u_j|d|m_j| = \sum_{j=1}^n E[|u_j||y_j|],
\]
where $y$ is the density of $m$, and thus, 
\[
\dom\rho_m=\{u\mid u_j\in L^1(\Omega,\F,|m_j|)\ \forall j=1,\dots,n\}.
\]
In this case, the integral is the Lebesgue integral.

\section{Topological duals of spaces of random variables}\label{sec:main}

This section presents the main results of the paper. The set up extends that of Banach function spaces by replacing the norm by an arbitrary collection of seminorms thus covering more general locally convex spaces of random variables. The main result identifies the topological dual of the space with the direct sum of the K\"othe dual and two spaces of singular functionals, the first of which is represented by finitely additive measures while the second is the orthogonal complement of $L^\infty$.

Let $\P$ be a collection of sublinear symmetric functions $p:L^1\to\ereals$, define
\[
\U:=\bigcap_{p\in\P}\dom p,
\]
and endow $\U$ with the locally convex topology generated by $\P$. Our aim is to characterize the topological dual $\U^*$ of $\U$. To this end, we will assume that
\begin{enumerate}[label={(A\arabic*)},ref=(A\arabic*)]
\item\label{ax1}
  the topology of $\U$ is no weaker than the relative $L^1$-topology,
\end{enumerate}
and that each $p\in\P$ satisfies
\begin{enumerate}[resume,label={(A\arabic*)},ref=(A\arabic*)]
\item\label{ax2}
  there exists a constant $c$ such that $p(u)\le c\|u\|_{L^\infty}$ for all $u\in L^1$,
\item\label{ax3}
  $p(u')\le p(u)$ for every $u\in\U$ and $u'\in L^1$ with $|u'|\le|u|$.
\end{enumerate}
Occasionally, we will also assume the following
\begin{enumerate}[resume,label={(A\arabic*)},ref=(A\arabic*)]
\item\label{ax4}
  $p(u1_{A^\nu})\downto 0$ for all $u\in L^\infty$ and decreasing sequence $(A^\nu)_{\nu=1}^\infty\subset\F$ with $P(A^\nu)\downto 0$.
\item\label{ax5}
  $p(u1_{A^\nu})\downto 0$ for all $u\in\U$ and decreasing sequence $(A^\nu)_{\nu=1}^\infty\subset\F$ with $P(A^\nu)\downto 0$.
\end{enumerate}
It is clear that \ref{ax4} and \ref{ax5} are implied by the following order continuity properties:
\begin{itemize}
\item[(A4')]\label{ax4b}
  $p(u^\nu)\downto 0$ for all $(u^\nu)\in L^\infty$ such that $|u^\nu|\downto 0$.
\item[(A5')]\label{ax5b}
  $p(u^\nu)\downto 0$ for all $(u^\nu)\in\U$ such that $|u^\nu|\downto 0$.
\end{itemize}

When $\P$ is a singleton, $\U$ is a normed space and we are in the setting of normed K\"othe function spaces; see e.g.~\cite{zaa67}. If, in addition, $p$ is lower semicontinuous in $L^1$, then $\U$ is a Banach function space; see Remark~\ref{rem:complete} below. In the Banach space setting, \ref{ax1} and \ref{ax2} hold under (A5') if $\U$ has a weak unit; see e.g.~\cite[Theorem~1.b.14]{lt79}. Necessity of the axioms will be discussed in more detail in Section~\ref{sec:nec}.

\begin{remark}\label{rem:ri}
Given $p\in\P$, we define, 
\begin{align*}
\hat\phi_p(t):=\sup_{A\in\F}\{p(1_A)\mid P(A)\le t\},\\
\check\phi_p(t):=\inf_{A\in\F}\{ p(1_A)\mid P(A)\ge t\}.
\end{align*}
Since we assume \ref{ax2} and \ref{ax3}, the condition \ref{ax4} is equivalent to 
\[
\lim_{t\searrow 0}\hat\phi_p(t)=0.
\]
If $\lim_{t\searrow 0}\check\phi_p(t)>0$, then $\U=L^\infty$. 

Assume now that $p$ is {\em rearrangement invariant} in the sense that $p(u)=p(\tilde u)$ whenever $u$ and $\tilde u$ have the same distribution. Then, for any $A\in\F$ with $P(A)=t$,
\[
\hat\phi_p(t) = \check\phi_p(t) = p(1_A)
\]
where the common value is known as the {\em fundamental function}. In particular, $\dom p= L^\infty$ if \ref{ax4} does not hold while \ref{ax4} is equivalent to $\lim_{t\searrow 0}\hat\phi_p(t)=0$.
\end{remark}
\begin{proof}
Assuming \ref{ax4}, let $t^\nu\searrow 0$. There exists $(A^\nu)_{\nu=1}^\infty$ such that $P(A^\nu)\le t^\nu$ and $\hat\phi_p(t^\nu)\le p(1_{A^\nu})+1/\nu$. Passing to a subsequence if necessary, $1_{A^\nu}\to 0$ almost surely. Defining $\hat A^\nu:=\bigcup_{\nu'\ge \nu} A^\nu$, $(\hat A^\nu)_{\nu=1}^\infty $ is decreasing with $A^\nu\subset\hat A^\nu$ and $P(\hat A^\nu)\searrow 0$, so, by \ref{ax3}--\ref{ax4}
\[
\hat \phi(t^\nu)\le p(1_{\hat A^\nu})+1/\nu\searrow 0.
\]
For the converse, let  $u\in L^\infty$ and $(A^\nu)_{\nu=1}^\infty\subset\F$ with $t^\nu:=P(A^\nu)\searrow 0$. By \ref{ax3}, $p(u1_A^\nu)\le \|u\|_{L^\infty} \hat\phi(t^\nu)\searrow 0$. 

If $\lim_{t\searrow 0}\check\phi_p(t)>\delta$ for some $\delta>0$, then $p(u) \ge p(\nu 1_{|u|\ge \nu})\ge \delta\nu$ whenever $P(\{|u|\ge \nu\})>0$, so $p(u)=+\infty$ if $u\notin L^\infty$.
\end{proof}

\begin{remark}\label{rem:complete}
As soon as \ref{ax1} holds, (relative) weak  compactness and sequential (relative) weak compactness on $\U$ are equivalent (Eberlein--Smulian property). If, in addition, $p$ are lower semicontinuous on $L^1$, then $\U$ is complete. In this case, $\U$ is a Banach/Fr\'echet space if $\P$ is a singleton/countable.

Denoting $p(u)=\rho(|u|)$, the function $p$ is lsc in $L^1$ if and only if $\rho$ has the Fatou property: for any sequence $(\eta^\nu)_{\nu=1}^\infty\subset L^1_+$ with $\eta \upto \eta\in L^1_+$, $\lim \rho(\eta^\nu)=\rho(\eta)$. 
\end{remark}

\begin{proof}
The first claim follows from the Theorem on p.~31 and Remark (2) on p.~39 in \cite{flo80}. If $(u^\nu)$ is a Cauchy net in $\U$, it is Cauchy also in $L^1$ so it $L^1$-converges to an $u\in L^1$. Being Cauchy in $\U$ means that for every $\epsilon>0$ and $p\in\P$, there is a $\bar\nu$ such that
\[
p(u^\nu-u^\mu)\le\epsilon\quad\forall\nu,\mu\ge\bar\nu.
\]
The lower semicontinuity then gives
\[
p(u^\nu-u)\le\epsilon\quad\forall\nu\ge\bar\nu
\]
so $u\in\U$, by triangle inequality, and $(u^\nu)$ converges in $\U$ to $u$. Thus $\U$ is complete.

If $p$ is lsc, $\liminf \rho(\eta^\nu)\ge \rho(\eta)$ while \ref{ax3} gives $\limsup \rho(\eta^\nu)\le \rho(\eta)$. If Fatou property holds and $u^\nu\to u$ in $L^1$, then, passing to a subsequence if necessary, $\bar \eta^\nu := \inf_{\nu'\ge\nu} |u^{\nu'}|$ increases pointwise to $|u|$, so $p(u)=\liminf \rho(\bar \eta^\nu)\le \liminf p(u^\nu)$.
\end{proof}


\begin{remark}\label{rem:decomp}
Under \ref{ax2} and \ref{ax3}, $\U$ is solid and decomposable. Solidity means that  $u\in\U$, $u'\in L^1$ and $|u'|\le|u|$ imply $u'\in\U$. Decomposability means that $u1_A+\bar u1_{\Omega\setminus A}\in\U$ for every $u\in\U$, $\bar u\in L^\infty$ and $A\in\F$.
\end{remark}

\begin{proof}
Assumption~\ref{ax2} implies that $L^\infty\subset\U$ while \ref{ax3} gives $u1_A\subset\U$ whenever $A\in\F$ and $u\in\U$. Since $\U$ is a linear space, the claim follows. 
\end{proof}

For each $p\in\P$, we define a sublinear symmetric function $p^\circ$ on $\M^1$ by
\[
p^\circ(m) := \sup_{u\in L^\infty}\left\{\int_\Omega udm\,\left.\right|\,p(u)\le 1\right\}.
\]

\begin{lemma}\label{lem:M}
Let $p\in\P$. For each $m\in\dom p^\circ$, every $u\in\dom p$ is $m$-integrable and 
\[
\int_\Omega udm\le p(u)p^\circ(m).
\]
For every $m\in\dom p^\circ$, there exist unique $y\in L^1\cap\dom p^\circ$ and $m^s\in\M^{1s}\cap\dom p^\circ$ such that 
\[
\int_\Omega udm = E[u\cdot y]+\int_\Omega udm^s\quad\forall u\in\dom p.
\]
Given $m^s\in\M^{1s}\cap\dom p^\circ$, there exists a decreasing $(A^\nu)_{\nu=1}^\infty\subset\F$ such that $P(A^\nu)\downto 0$ and 
\[
\int u1_{\Omega\backslash A^\nu} dm^s=0
\]
for every $u\in\dom p$. Under \ref{ax4}, $\M^{1s}\cap\dom p^\circ=\{0\}$.
\end{lemma}


\begin{proof}
Lemma~\ref{lem:ext} and \ref{ax3} give
\[
\int_\Omega udm\le\rho_m(u) \le \sup_{u'\in L^\infty}\{\int_\Omega u'dm \mid |u'|\le |u|\}\le p(u)p^\circ(m).
\]
By Theorem~\ref{thm:YH}, there exist $y\in L^1$ and $m^s\in(L^\infty)^s$ such that $m=yP+m^s$. Let $\alpha<p^\circ(y)$ and $\alpha^s<p^\circ(m^s)$ and $u,u^s\in L^\infty$ such that $p(u),p(u^s)\le 1$ and
\[
\int_\Omega uydP\ge\alpha\quad\text{and}\quad\int_\Omega u^sdm^s\ge\alpha^s.
\]
Let $(A^\nu)_{\nu=1}^\infty\subset\F$ be decreasing with $P(A^\nu)\downto 0$ and $m^s(\Omega\setminus A^\nu)=0$ and let $u^\nu=\lambda u1_{\Omega\setminus A^\nu}+(1-\lambda)u^s1_{A^\nu}$, where $\lambda\in(0,1)$. By convexity and \ref{ax3},
\[
p(u^\nu)\le\lambda p(u1_{\Omega\setminus A^\nu})+(1-\lambda)p(u^s1_{A^\nu})\le\lambda p(u)+(1-\lambda)p(u^s)\le 1
\]
while
\[
\limsup\int_\Omega u^\nu dm \ge\lambda\alpha+(1-\lambda)\alpha^s.
\]
Thus, $p^\circ(m)\ge\lambda\alpha+(1-\lambda)\alpha^s$. Since $\alpha<p^\circ(y)$ and $\alpha^s<p^\circ(m^s)$ were arbitrary, $p^\circ(m)\ge\lambda p^\circ(y)+(1-\lambda)p^\circ(m^s)$. Since $\lambda\in(0,1)$ was arbitrary, we get $p^\circ(y)\le p^\circ(m)$ and $p^\circ(m^s)\le p^\circ(m)$. Thus, $y\in\dom p^\circ$ and $m^s\in\dom p^\circ$.

To prove the last claim, let $m^s\in\M^{1s}\cap\dom p^\circ$. By the first claim,
\[
\int_\Omega u1_Adm^s\le p(u1_A)p^\circ(m^s)\quad\forall u\in L^\infty,A\in\F
\]
so, by the last claim of Theorem~\ref{thm:YH}, condition \ref{ax4} implies $m^s=0$.
\end{proof}

Let $\M$ be the set of $P$-absolutely continuous finitely additive measures $m$ such that $p^\circ(m)<\infty$ for some $p\in\P$. The set of purely finitely additive elements of $\M$ will be denoted by $\M^s$. The set of densities $y=dm/dP$ of countably additive $m\in\M$ will be denoted by $\Y$.

The following is the main result of this section. It identifies the topological dual of $\U$ with the direct sum of the K\"othe space, purely finitely additive measures $\M^s$ and the annihilator
\[
(L^\infty)^\perp :=\{w\in \U^*\mid \langle u,w\rangle =0\ \forall u\in L^\infty\}
\]
of $L^\infty$. 
\begin{theorem}\label{thm:f2}
We have
\[
\U^* = \Y\oplus\M^s\oplus (L^\infty)^\perp
\]
in the sense that for every $u^*\in\U^*$ there exist unique $y\in\Y$, $m^s\in\M^s$ and $w\in (L^\infty)^\perp$ such that
\[
\langle u,u^*\rangle = E[u\cdot y]+\int_\Omega u dm^s + \langle u,w\rangle.
\]
For every $u\in\U$ and $m\in\M$,
\[
\int_\Omega udm \le p(u)p^\circ(m).
\]
Given $w\in(L^\infty)^\perp$ and $u\in\U$,  there exists a decreasing sequence $(A^\nu)_{\nu=1}^\infty\subset\F$ with $P(A^\nu)\searrow 0$ and
\[
\langle u,w\rangle = \langle u1_{A^\nu},w\rangle\quad\forall\nu=1,2,\ldots.
\]
Under \ref{ax4}, $\M^s=\{0\}$ and under \ref{ax5}, $(L^\infty)^\perp=\{0\}$.
\end{theorem}

\begin{proof}
By Lemma~\ref{lem:M}, $\M\subset\U^*$, so $\M\oplus (L^\infty)^\perp\subseteq \U^*$. To prove the opposite inclusion, let $u^*\in\U^*$. There exists $p\in\P$ and $\gamma>0$ such that $u^*\le \gamma p$. Assumption~\ref{ax2} implies that $u^*$ is continuous in $L^\infty$. By Theorem~\ref{thm:YH}, there exists a unique $m\in\M^1$ such that $ \langle u,u^*\rangle = \int_\Omega udm$ for all $u\in L^\infty$. Since $u^*\le \gamma p$, we have $m\in\dom p^\circ$, so $m$ is continuous on  $\U$ by Lemma~\ref{lem:M}. Now $w := u^*-m$ belongs to $(L^\infty)^\perp$, so $u^*$ has the required decomposition. Given another decomposition $u^*=\tilde m +\tilde w$ with $\tilde w\in (L^\infty)^\perp$ and $\tilde m\in\M$, we have $(m-\tilde m) +(w-\tilde w)=0$. Thus $\int_\Omega ud(m-\tilde m) =0$ for all $u\in L^\infty$, so $m-\tilde m=0$ and hence also $w-\tilde w=0$, so the decomposition is unique.

The inequality follows directly from that of Lemma~\ref{lem:M}. Let $u\in\U$ and $A^\nu :=\{|u|>\nu\}$. Clearly $P(A^\nu)\searrow 0$ and $u1_{\Omega\setminus A^\nu}\in L^\infty$, so $\langle u1_{\Omega\setminus A^\nu},w\rangle=0$ and thus $w$ is singular. That $\M=\Y$ under \ref{ax4} is the last claim of Lemma~\ref{lem:M}. Under \ref{ax5}, the truncations $u^\nu:=u1_{\{|u|\ge\nu\}}$ of any $u\in\U$ converge to $u$ so $L^\infty$ is dense in $\U$ and thus, $(L^\infty)^\perp=\{0\}$. 
\end{proof}
Applications of Theorem~\ref{thm:f2} are given in Section~\ref{sec:ex}. When $\P$ is a singleton, we are in the setting of \cite{zaa67}, where the dual of $\U$ is decomposed into a direct sum of $\Y$ and "singular elements". Theorem~\ref{thm:f2} gives  a more precise description of the singular elements as a direct sum of $\M^s$ and $(L^\infty)^\perp$. 

Note that the inequality in Theorem~\ref{thm:f2} implies that $p^\circ$ coincides on $\M$ with the polar (i.e., the dual seminorm) of $p$. An application of Theorem~\ref{thm:f2} and the Hahn-Banach theorem gives the following result, where $\tilde\U$ is the closure of $L^\infty$ in $\U$.

\begin{corollary}\label{cor:f2}
We have
\[
\tilde \U^* = \M
\]
in the sense that for every $\tilde u^*\in\tilde\U^*$ there exist unique $m\in\M$ such that
\[
\langle \tilde u,\tilde u^*\rangle = \int_\Omega \tilde u dm.
\]
In particular, if \ref{ax4} holds, then $\tilde\U^*=\Y$ and if \ref{ax5} holds, then $\tilde\U=\U$.
\end{corollary}

The following lists some basic properties of the K\"othe dual $\Y$.

\begin{lemma}\label{lem:polar}
We have
\begin{enumerate}
\item
$L^\infty\subseteq\Y$
\end{enumerate}
and, for each $p\in\P$
\begin{enumerate}[resume]
\item
there is a constant $c$ such that $c\|y\|_{L^1}\le p^\circ(y)$ for all $y \in L^1$,
\item
  $p^\circ(y')\le p^\circ(y)$ for every  $y',y\in L^1$ with $|y'|\le|y|$.
\item
  We have the ``H\"older's inequality''
\[
E[u\cdot y]\le p(u)p^\circ(y)
\]
and, conversely, if there is $c>0$ such that $c\|u\|_{L^1}\le p(u)$ for all $u$ and $p$ is lsc in $L^1$, then $p^\circ(y)<\infty$ whenever $E[u\cdot y]<\infty$ for all $u\in\dom p$.
\end{enumerate}
In particular,  $\Y$ is solid and decomposable.
\end{lemma}

\begin{proof}
Assumption \ref{ax1} implies 1, and \ref{ax2} implies 2.
By \ref{ax3},
\begin{align*}
p^\circ(y')&= \sup_{u'\in L^\infty,u\in L^\infty} \left\{E [u'\cdot y']\midb |u'|\le|u|,\ p(u)\le 1\right\}\\
&= \sup_{u\in L^\infty} \left\{E [|u||y'|]\midb  p(u)\le 1\right\}\\
&\le \sup_{u\in L^\infty} \left\{E [|u||y|]\midb  p(u)\le 1\right\}\\
&=p^\circ(y),
\end{align*}
so 3 holds. To prove 4, the inequality in Lemma~\ref{lem:M} gives the H\"older's inequality. Assume now that $p^\circ(y)=+\infty$. Let $\alpha^\nu > 0$ be such that $\sum \alpha^\nu =1$. There exists $u^\nu$ with $p(u^\nu)\le 1$, $u^\nu\cdot y \ge 0$ and $E[u^\nu\cdot y]\ge 1/\alpha^\nu$. We have that $\sum_{\nu'=1}^\nu \alpha^{\nu'} u^{\nu'}$ converges to $u:=\sum \alpha^\nu u^\nu$ in $L^1$ and, since $p$ is lsc in $L^1$, $u \in \dom p$. By monotone convergence,
\[
E [u\cdot y] = \sum_{\nu=1}^\infty \alpha^\nu E[u^\nu\cdot y] = +\infty,
\]  
which completes the proof.
\end{proof}

\section{Examples}\label{sec:ex}

The following example is a direct application of Corollary~\ref{cor:f2}.

\begin{example}[The space of finite moments]\label{ex:fm}
The $L^p$-norms with $p\ge 1$ satisfy \ref{ax1}-\ref{ax5}, so, given an increasing sequence $S\subset[1,\infty)$, the space 
\[
\U:=\bigcap_{p\in S}L^p
\]
is Fr\'echet space and its dual may be identified with  
\[
\Y:=\bigcup_{p\in S}L^p
\]
under the bilinear form $\langle u,y\rangle=E[u\cdot y]$. When $S$ is unbounded, $\U$ is the space of measurable functions with finite moments while if $\sup S=p$ with $p\notin S$, $\U$ is the space of measurable functions with moments strictly less than $p$.
\end{example}

Given a set $C$ in a linear space, we will use the notation
\[
\pos C := \bigcup_{\alpha>0}(\alpha C)\quad\text{and}\quad
C^\infty := \bigcap_{\alpha>0}(\alpha C).
\]
The following construction, inspired by the Luxemburg norm in the theory of Orlicz spaces, turns out to be convenient. 
\begin{example}\label{ex:lux}
Let $H:L^1\to\ereals_+$ be lsc convex such that $H(0)=0$ and
\begin{enumerate}[label =  (H\arabic*), ref=(H\arabic*)]
\item\label{H1} there is a constant $c>0$ such that $H(u)\le 1$ implies $\|u\|_{L^1}\le c$,
\item\label{H2} $L^\infty\subset\pos (\dom H)$,
\item\label{H3} $H(u_1)\le H(u_2)$ whenever $|u_1|\le|u_2|$.
\end{enumerate}
The function
\[
p(u):=\inf\{\beta>0\mid H(u/\beta)\le 1\}
\]
is lsc, symmetric and sublinear. Let $\P=\{p\}$ and $\U=\dom p$. Assumptions \ref{ax1}--\ref{ax3} hold and, in particular, $\U$ is a Banach space with dual
\[
\U^*=\M\oplus (L^\infty)^\perp,
\]
where 
\[ 
\M=\pos\dom H^*
\]
with $H^*:\M^1\to\ereals$ given by
\[
H^*(m):=\sup_{u\in L^\infty}\{\int_\Omega udm-H(u)\}.
\]
For any $m\in\M^1$,
\[
p^\circ(m)=\sup_{u\in L^\infty}\{\int_\Omega udm\mid H(u)\le 1\} = \inf_{\beta>0}\{\beta H^*(m/\beta)+\beta\},
\]
restriction of $p^\circ$ to $\M$ is the polar of $p$ and
\[
\|m\|_{H^*}\le p^\circ(m)\le 2\|m\|_{H^*},
\]
where
\[
\|m\|_{H^*} :=\inf\{\beta>0\mid H^*(m/\beta)\le 1\}.
\]
Assume now that $L^\infty\subseteq\dom H$. If
\begin{enumerate}[resume,label =  (H\arabic*), ref=(H\arabic*)]
\item\label{H4} $H(u^\nu)\downto 0$ whenever $(u^\nu)_{\nu=1}^\infty\subset L^\infty$ with $|u^\nu|\downto 0$ almost surely,
\end{enumerate}
then \ref{ax4} holds so $\M^s=\{0\}$ and the dual of the closure $\tilde\U$ of $L^\infty$ in $\U$ can be identified with $\Y$. If
\begin{enumerate}[resume,label =  (H\arabic*), ref=(H\arabic*)]
\item\label{H5} $H(u^\nu)\downto 0$ whenever $(u^\nu)_{\nu=1}^\infty\subset \dom H$  with $|u^\nu|\downto 0$ almost surely,
\end{enumerate}
then $\tilde\U = (\dom H)^\infty$. In particular, $\U=\tilde \U$ if $\dom H$ is a cone.
\end{example}

\begin{proof}
Let $u^\nu\to u$ in $L^1$ be such that $p(u^\nu)\le\alpha$ or, in other words, $H(u^\nu/\alpha)\le 1$. Thus lower semicontinuity of $H$ implies that of $p$. It is clear that \ref{H1} implies \ref{ax1}. By \ref{H2},  $p$ is finite on $L^\infty$. Since $p$ is lsc on $L^1$, it is lsc on $\sigma(L^\infty,L^1)$. Thus, by \cite[Corollary~8B]{roc74}, $p$ is continuous in $L^\infty$ and thus \ref{ax2} holds. Assumption~\ref{ax3} is clear from \ref{H3}.

Let $m\in\M^1$. Since the infimum in the definition of the Luxemburg norm is attained,
\begin{align*}
p^\circ(m) &=\sup_{u\in L^\infty}\{\int_\Omega udm \,|\,p(u)\le 1\}=\sup_{u\in L^\infty}\{\int_\Omega udm\mid H(u)\le 1\}.
\end{align*}
Lagrangian duality gives
\begin{align*}
p^\circ(m) &= \inf_{\beta>0}\sup_{u\in L^\infty}\{\int_\Omega udm-\beta H(u)+\beta\} = \inf_{\beta>0}\{\beta H^*(m/\beta)+\beta\}.
\end{align*}

Clearly,
\[
p^\circ(m)\le\inf_{\beta>0}\{\beta H^*(m/\beta)+\beta\mid H^*(m/\beta)\le 1\}\le 2\inf\{\beta>0\mid H^*(m/\beta)\le 1\}.
\]
On the other hand, we have
\[
p^\circ(m) = \inf_{\beta>0}\{\beta H^*(m/\beta)+\beta\} = \inf_{\alpha>0}\frac{g(\alpha m)}{\alpha},
\]
where $g(m)=H^*(m)+1$. Since $H^*\ge 0$, we have $g\ge\|\cdot\|_{H^*}$ when $\|m\|_{H^*}\le 1$. When $\|m\|_{H^*}>1$, convexity and the fact that $H^*(0)=0$ give
\[
H^*(m/\|m\|_{H^*})\le H^*(m)/\|m\|_{H^*}.
\]
By definition of $\|m\|_{H^*}$, the left side equals $1$ so $\|m\|_{H^*}\le H^*(m)\le g(m)$. Thus,
\[
p^\circ(m) \ge \inf_{\alpha>0}\frac{\|\alpha m\|_{H^*}}{\alpha} = \|m\|_{H^*}.
\]

If \ref{H4} holds and $|u^\nu|\downto 0$ almost surely in $L^\infty$, then for all $\beta>0$,
\[
H(u^\nu/\beta)\downto 0
\]
so $p(u^\nu)\downto 0$. In particular, \ref{ax4} holds.

To prove the last claim, let $u\in(\dom H)^\infty$, $u^\nu:= u\one_{|u|\le\nu}$ and $\beta>0$. By \ref{H3}, $u-u^\nu=u1_{\Omega\setminus \{|u|\le \nu\}}\in\beta\dom H$ so \ref{H5} implies
\[
H((u-u^\nu)/\beta)\downto 0.
\]
Since $\beta>0$ was arbitrary, we get $p(u-u^\nu)\downto 0$ so $(\dom H)^\infty\subseteq \tilde \U$. To prove the converse, it remains to show that $(\dom H)^\infty$ is closed in $\U$. If $(u^\nu)$ is in $(\dom H)^\infty$ and converges to $u\in \tilde\U$, we have for any $\beta>0$, 
\begin{align*}
H(u/(2\beta))\le \frac{1}{2}H(u^\nu/\beta)+\frac{1}{2}H((u-u^\nu)/\beta) \le \frac{1}{2}H(u^\nu/\beta)+\frac{1}{2}
\end{align*}
for $\nu$ large enough, so $H(u/2\beta)<\infty$ and thus $u\in(\dom H)^\infty$.
\end{proof}

Musielak--Orlicz spaces are generalizations of Orlicz spaces where the associated Young function $\Phi$ is allowed to be random in the sense that it is a function on $\reals\times\Omega$ such that
\[
\omega\mapsto\{(\xi,\alpha)\mid \Phi(\xi,\omega)\le\alpha\}
\]
is a convex-valued measurable mapping; see \cite[Chapter~14]{rw98}. If $\Phi$ only takes finite real values, this happens exactly when $\Phi(\xi,\cdot)$ is measurable for every $\xi\in\reals$ and $\Phi(\cdot,\omega)$ is convex for every $\omega\in\Omega$. The dual of a Musielak--Orlicz space can be characterized in terms of the conjugate function defined by
\[
\Phi^*(\eta,\omega) = \sup_{\xi\in\reals}\{\xi\eta - \Phi(\xi,\omega)\}.
\]
The measurability condition on $\Phi$ implies the same property for $\Phi^*$; see \cite[Theorem~14.50]{rw98}.

\begin{example}[Musielak-Orlicz spaces]\label{ex:mo}
Let $\Phi:\reals\times\Omega\to\ereals_+$ be nonzero random symmetric convex function with $\Phi(0)=0$ and such that $\Phi(a,\cdot),\Phi^*(a,\cdot)\in L^1$ for some constant $a>0$. Endowed with the Luxemburg norm
\[
\|u\|_\Phi :=\inf\{\beta>0 \mid E\Phi(|u|/\beta)\le 1\},
\]
$L^\Phi:=\{u\in L^1\mid \|u\|_\Phi<\infty\}$ is a Banach space. The dual of $L^\Phi$ is
\[
(L^\Phi)^*=L^{\Phi^*}\oplus \M^s\oplus(L^\infty)^\perp,
\]
where
\[
\M^{s}=\{m\in M^{1s}\mid \sigma_\Phi(m)<\infty\}
\]
with $\sigma_\Phi(m) :=\sup_{u\in L^\infty}\{ \int_\Omega udm \mid E\Phi(|u|)<\infty\}$. For any $y+m^s \in L^{\Phi^*}\oplus\M^s$, the dual norm can be expressed as
\begin{align*}
\|y+m^s\|_{\Phi}^* &= \sup_{u\in L^\infty}\{E[u\cdot y]+\int_\Omega udm^s\mid E\Phi(|u|)\le 1\}\\
 &= \inf_{\beta>0}\{\beta E\Phi^*(|y|^*/\beta)+\beta\}+\sigma_{\Phi}(m^s),
\end{align*}
we have
\[
\|y\|_{\Phi^*}\le \|y\|_{\Phi}^*\le 2\|y\|_{\Phi^*}\quad\forall y\in L^{\Phi^*},
\]
and the dual of the closure $M^\Phi$ of $L^\infty$ in $L^\Phi$ is
\[
(M^\Phi)^*=L^{\Phi^*}\oplus\M^s.
\]

Assume now that $\Phi(a,\cdot)\in L^1$ for all $a>0$. Then, $\M^s=\{0\}$, $M^\Phi$ coincides with the Morse heart 
\[
(\dom E\Phi)^\infty = \{\xi\in L^1 \mid E\Phi(|\xi|/\beta)< \infty\quad \forall\beta>0\},
\]
and, in particular, $L^\Phi=M^\Phi$ if $\dom E\Phi$ is a cone. 
\end{example}

\begin{proof}
We apply Example~\ref{ex:lux} to $H(u):=E\Phi(|u|)$. By \cite[Theorem~14.60]{rw98}, 
\[
H(u) = \sup_{\eta\in L^\infty}E\{[|u|\eta] - \Phi^*(\eta)\},
\]
so $H$ is $L^1$-lsc. This also gives
\[
H(u)\ge a\|u\|_{L^1} - E\Phi^*(a)
\]
so $\Phi^*(a)\in L^1$ implies \ref{H1}. The assumption $\Phi(a)\in L^1$ implies that $H(u)<\infty$ when $\|u\|_{L^\infty}\le a$ so \ref{H2} holds. Property \ref{H3} holds since $\Phi$ is increasing. By \cite[Theorem~1]{roc71} and \cite[Theorem~15.3]{roc70a},
\[
H^*(m)=\sup_{u\in L^\infty}\{\int udm-Eh(u)\}= E\Phi^*(|y|^*)+\sigma_\Phi(m^s).
\]
If $\Phi(a)\in L^1$ for all $a>0$, then $L^\infty\subset\dom H$ and \ref{H4} and \ref{H5} hold by monotone convergence theorem. Thus all the claims follow from Example~\ref{ex:lux}.
\end{proof}
In \cite{ml19}, the assumption $\Phi(a,\cdot)\in L^1$ for all $a>0$ is called "local integrability". Thus we recover \cite [Theorem 2.4.4]{ml19} for probability spaces. Our characterization of the dual without local integrability seems new.

\begin{example}[Risk measures]\label{ex:rm}
Let $\rho:L^1\to\ereals$ be a ``convex risk measure'' in the sense that it is 
convex, nondecreasing, $\rho(0)=0$ and $\rho(\xi+\alpha)=\rho(\xi)+\alpha$ for all $\xi\in L^1$ and $\alpha\in\reals$. Assume that $n=1$, $\rho$ is $L^1$-lsc and that there is a constant $c>0$ such that $\rho(|u|)\le 1$ implies $\|u\|_{L^1}\le c$.

Endowed with the norm
\[
\|u\|_\rho := \inf\{\beta>0\mid\rho(|u|/\beta)\le 1\},
\]
$L^\rho:=\{u\in L^1\mid \rho(|u|)<\infty\}$ is a Banach space whose dual can be identified with $\M\oplus (L^\infty)^\perp$, where
\[ 
\M=\{m\in\M^1 \mid \exists \beta>0: \alpha(|m|/\beta)<\infty\}
\]
with $\alpha:\M^1\to\ereals$ defined by
\[
\alpha(m):=\sup_{\xi\in L^\infty_+}\{\int_\Omega \xi dm -\rho(\xi)\}.
\]
For any $m\in\M$, the dual norm can be expressed as
\[
\|m\|_\rho^* = \sup_{u\in L^\infty}\{\int_\Omega udm \mid \rho(u)\le 1\} = \inf_{\beta>0}\{\beta \alpha(|m|/\beta)+\beta\},
\]
and
\[
\|m\|_{\alpha}\le \|m\|_\rho^*\le 2\|m\|_{\alpha},
\]
where
\[
\|m\|_{\alpha} :=\inf\{\beta>0\mid \alpha(|m|/\beta)\le 1\}.
\]

\begin{enumerate}
\item If $\rho$ has the Lebesgue property on $L^\infty$: $\rho(\xi^\nu)\downto 0$ for any decreasing sequence $(\xi^\nu)\subset L^\infty$ with $\xi^\nu\downto 0$ almost surely,
\end{enumerate}
 then the dual of the closure $\tilde L^\rho$ of $L^\infty$ in $L^\rho$ can be identified with 
\[
L^{\alpha}:=\{y\in L^1 \mid \exists \beta>0: \alpha(|y|/\beta)<\infty\}.
\]

\begin{enumerate}[resume]
\item If $\rho$ has the Lebesgue property on $\dom\rho$: $\rho(\xi^\nu)\downto 0$ for any decreasing sequence $(\xi^\nu)\subset\dom\rho$ with $\xi^\nu\downto 0$ almost surely,
\end{enumerate}
 then 
\[
\tilde L^\rho= \{u\in L^1\mid \rho(|u|/\beta)<\infty\ \forall\beta>0\},
\] 
and, in particular, $L^\rho=\tilde L^\rho$ if $\dom\rho$ is a cone.

\end{example}

\begin{proof}
We apply Example~\ref{ex:lux} to the function $H(u):=\rho(|u|)$. By assumption, \ref{H1} and \ref{H3} hold. By monotonicity and translation invariance, $\rho(|u|)\le \rho(\|u\|_{L^\infty})=\|u\|_{L^\infty}$, so $L^\infty\subset \dom H$. In particular, \ref{H2} holds. The conditions \ref{H4} and \ref{H5} in Example~\ref{ex:lux}  translate directly to those of 1 and 2. Thus the claims follow from Example~\ref{ex:lux}, since here
\begin{align*}
H^*(m) &:=\sup_{u\in L^\infty}\{\int udm-\rho(|u|)\}\\
&=\sup_{u\in L^\infty,\xi\in L^\infty_+}\{\int u\xi dm-\rho(\xi) \mid |u|=1\}\\
&=\sup_{\xi\in L^\infty_+}\{\int \xi d|m|-\rho(\xi)\}\\
&=\alpha(|m|),
\end{align*}
where the second last equality follows from \cite[Theorem 2.3]{yh52} and the fact that $\nu(A):=\int_A \xi dm$ is a finitely additive measure with $|\nu|(A)= \int_A\xi d|m|$.
%
\end{proof}

Given $u\in L^1$, let
\[
n_u(\tau):=E1_{\{|u|>\tau\}}
\]
and
\[
q_u(t):=\inf\{\tau\in\reals\mid n_u(\tau)\le t\}.
\]
Note that $\tau\mapsto 1-n_u(\tau)$ is the cumulative distribution function of $|u|$ and that $q_u$ is an inverse of $n_u$. Both $n_u$ and $q_u$ are nonincreasing.

\begin{lemma}\label{lem:cvar}
We have
\[
\int_0^t q_u(t)dt=\inf_{s\in\reals^+}\{ts + E[|u|-s]^+\}.
\]
\end{lemma}

\begin{proof}
By Theorems~23.5 and 24.2 of \cite{roc70a}, the functions
\[
f(t) := \int_0^tq_u(s)ds
\]
and
\[
f^*(s) = \int_0^sn_u(\tau)d\tau - \int_0^\infty n_u(\tau)d\tau = -\int_s^\infty n_u(\tau)d\tau
\]
are concave and conjugate to each other. By Fubini,
\[
f^*(s) = -E\int_s^\infty 1_{\{|u|>\tau\}}d\tau = -E[|u|-s]^+
\]
so
\[
\int_0^tq_u(s)ds =\inf_{s\in\reals^+}\{ts + E[|u|-s]^+\},
\]
 by the biconjugate theorem (see e.g.~\cite[Theorem~12.2]{roc70a}).
\end{proof}
Recall that a probability space is {\em resonant} if it is atomless or completely atomic with all atoms having equal measure.

\begin{remark}
Assume that $(\Omega,\F,P)$ is resonant. Every rearrangement invariant seminorm $p$ satisfies the ``Jensen's inequality''
\[
p(E^\G u)\le p (u)  \quad\forall u\in\U
\]
for every sub sigma-algebra $\G\subset\F$. Here $E^\G u$ is the conditional expectation of $u$. However, rearrangement invariance is not necessary for the Jensen's inequality to hold for every sub-$\sigma$-algebra $\G$.
\end{remark}

\begin{proof}
By Jensen's inequality, $E|E^\G u|\le E|u|$, so Lemma~\ref{lem:cvar} implies 
\[
\int_0^t q_{E^\G u}(s) ds\le \int_0^t q_{u}(s) ds.
\]
By Hardy's lemma (\cite[Proposition 2.3.6]{bs88}),
\[
\int q_{E^\G u}(s) q_y(s) \le \int q_{u}(s) q_y(s)
\]
for any $y\in L^1$. Thus the claim follows from \cite[Corollary 2.4.4]{bs88}.

As to the necessity, let $\F=\{\emptyset,A,A^C,\Omega\}$, where $P(A)=P(A^C)=1/2$. Then the only strict sub-$\sigma$-algebra $\G$ of $\F$ is the trivial one. Let 
\[
p(u) :=\max\{ E|u|, E[1_A |u|/P(A)\}.
\]
Note that $ E[1_A |E^\G u|/P(A) =E[E^\G 1_A |E^\G u|/P(A) =E[|E^\G u|$ so that 
\begin{align*}
p(E^\G \eta) &= E|E^\G u| \le E|u| \le p(u),
\end{align*}
and $p$ satisfies the $\G$-conditional Jensen's inequality for every $\G\subset \F$. However, $p(1_A)=1$ while $p(1_{A^C})=1/2$, so $p$ is not rearrangement invariant.
\end{proof}

\begin{example}[Lorentz and Marcinkiewicz spaces]\label{ex:ml}
Assume that $(\Omega,\F,P)$ is resonant. Given a nonnegative concave increasing function $\phi$ on $[0,1]$ with $\phi(0)=0$, the associated {\em Marcinkiewicz space} is the linear space $M_\phi$ of $u\in L^1$ with
\[
\|u\|_\phi := \sup_{t\in(0,1]}\left\{\frac{1}{\phi(t)}\int_0^tq_u(s)ds\right\} <\infty.
\]
The function $\|\cdot\|_\phi$ is a norm and $M_\phi$ is a Banach space. If $\lim_{t\searrow 0}t/\phi(t)>0$, we have $M_\phi=L^\infty$. Assume now that $\lim_{t\searrow 0}t/\phi(t)=0$. The topological dual of $M_\phi$ is
\[
M_\phi^*=\Lambda_\Phi\oplus (L^\infty)^\perp,
\]
where $\Lambda_\Phi$ is the {\em Lorentz space}
\begin{align*}
\Lambda_\phi:=\{y\in L^1\mid \|y\|^*_\phi<\infty \},
\end{align*}
where 
\[
\|y\|^*_\phi := \int_0^1q_y(t)d\phi(t).
\]
The closure of $L^\infty$ in $M_\phi$ can be expressed as
\[
\tilde M_\phi =\{u\in L^1\mid \lim_{t\searrow 0} \frac{1}{\phi(t)}\int_0^t q_u(s)ds=0\}.
\]
The topological dual of  $\tilde M_\phi$ is $\Lambda_\Phi$ and the topological dual of $\Lambda_\phi$ is $M_\phi$.
\end{example}

\begin{proof}
By Lemma~\ref{lem:cvar},
\[
u\mapsto \int_0^t q_u(t)dt
\]
is the infimal projection of a sublinear function of $s$ and $u$ and thus, sublinear in $u$. It is also continuous in $L^1$. It follows that $\|\cdot\|_\phi$ is sublinear, symmetric and lsc in $L^1$.

Since
\[
\|u\|_\phi\ge \phi(1)\int_0^1 q_u(s)ds =\phi(1)E[|u|],
\]
\ref{ax1} holds. By Remark~\ref{rem:complete}, $M_\phi$ is Banach. Since $q_u\le\|u\|_{L^\infty}$, we have
\[
\|u\|_\phi\le \sup_{t\in(0,1]}\frac{t}{\phi(t)} \|u\|_{L^\infty},
\]
where $\sup_{t\in(0,1]}\frac{t}{\phi(t)}<\infty$ since $\phi$ is concave and strictly positive for $t>0$. Thus, \ref{ax2} holds. Property \ref{ax3} is clear. Given $A\in\F$, 
\[
\|1_A\|_\phi = \sup_t \frac{1}{\phi(t)} \min\{t,P(A)\} = \frac{P(A)}{\phi(P(A))},
\]
since $t\mapsto \frac{t}{\phi(t)}$ is increasing by concavity. Thus $\hat\phi_p(t):=\frac{t}{\phi(t)}$ is the fundamental function of $M_\phi$. By Remark~\ref{rem:ri},  $M_\phi=L^\infty$ if $\lim_{t\searrow 0}t/\phi(t)>0$ while \ref{ax4} holds if $\lim_{t\searrow 0}t/\phi(t)=0$.
 We have
\begin{align*}
\|y\|^*_\phi &= \sup_{u\in L^1}\{E[uy]\mid\|u\|_\phi\le 1\}\\
&= \sup_{u\in L^1}\{\int_0^1 q_u(t)q_y(t)dt\mid \int_0^tq_u(s)ds\le\phi(t)\ \forall t\in[0,1]\}\\
&=\int_0^1q_y(t)\phi'(t)dt\\
&= \int_0^1q_y(t)d\phi(t),
\end{align*}
where the second equality follows from \cite[Corollary 2.4.4]{bs88} and the third from Hardy's lemma \cite[Proposition~2.3.6]{bs88}. The representation of the topological dual of $M_\phi$ now follows from Theorem~\ref{thm:f2}.


If $u\in L^\infty$, $q_u$ is bounded, so  
\[
\lim_{t\searrow 0} \frac{1}{\phi(t)}\int_0^tq_u(s)ds = \lim_{t\searrow 0} \frac{t}{\phi(t)}\frac{1}{t}\int_{[0,t]} q_u(s)ds =0,
\]
by assumption. Thus, $L^\infty\subset\tilde M_\phi$. Let $u\in M_\phi$ and $\tilde M_\phi$. We have $q_{u+\tilde u}(s^1+s^2)\le q_{u}(s^1)+q_{\tilde u}(s^2)$, so
\begin{align*}
\lim_{t\searrow 0} \frac{1}{\phi(t)}\int_0^t q_{u}(s)ds &\le \lim_{t\searrow 0} \frac{1}{\phi(t)}\int_0^t (q_{u-\tilde u}(s/2) + q_{\tilde u}(s/2))ds\\
&=\lim_{t\searrow 0} \frac{1}{\phi(t)}\int_0^t q_{u-\tilde u}(s/2)ds\\
&=\lim_{t\searrow 0} \frac{2}{\phi(t)}\int_0^{2t} q_{u-\tilde u}(s)ds\\
&\le\lim\frac{1}{\phi(2t)}\int_0^{2t} q_{u-\tilde u}(s)ds\\
&\le\|u-\tilde u\|_{\phi},
\end{align*}
where the second last inequality follows from concavity of $\phi$. Thus, $\tilde M_\phi$ is closed in $M_\phi$ so $\tilde M_\phi$ contains the closure of $L^\infty$. To prove the converse, let $u\in\tilde M_\phi$ and $u^\nu=u1_{\{|u|\le \nu\}}$. We have $q_{u-u^\nu}(t)=0$ for $t\ge t^\nu :=P(|u|\ge \nu)$ while $q_{u-u^\nu}(t)=q_u(t)$ for $t<t^\nu$. Thus,
\begin{align*}
\|u-u^\nu\|_\phi &=\sup_{t\in[0,1]}\left\{\frac{1}{\phi(t)}\int_0^tq_{u-u^\nu}(s)ds\right\} = \sup_{t\in[0,t^\nu]}\left\{\frac{1}{\phi(t)}\int_0^tq_u(s)ds\right\}.
\end{align*}
Since $u\in\tilde M_\phi$, this converges to $0$ as $\nu\to\infty$. Thus, $\tilde M_\phi$ is the closure of $L^\infty$ in $M_\phi$.

By Lemma~\ref{lem:polar}, the Lorentz seminorm satisfies \ref{ax1}-\ref{ax3}. If $y^\nu\downto 0$ with $\|y^\nu\|_\phi^*<\infty$, we have $q_{y^\nu}\downto 0$, so by monotone convergence, $\|y^\nu\|_\phi^*\downto 0$. Thus, the Lorenz norm satisfies \ref{ax5}. The fact that the topological dual of $\Lambda_\phi$ is $M_\phi$ now follows from Theorem~\ref{thm:f2} and the fact that, by the bipolar theorem, $p$ is the polar of $p^\circ$.
\end{proof}

Much like in Example~\ref{ex:fm}, one can characterize topological duals of locally convex (resp.\ Fr\'echet) spaces obtained by intersecting Markinkiewich spaces associated with a (resp\ countable) collection of nonnegative concave increasing functions $\phi$.

\begin{example}[Generalized Orlicz-spaces]
Let $\Phi$ be as in Example~\ref{ex:mo} with $\dom\Phi=\reals$ and let $r$ be a sublinear symmetric lsc function on $L^1$ satisfying \ref{ax1}--\ref{ax4}. Endowed with the norm
\[
\|u\|_{\Phi,r} := \inf\{\beta>0\mid r(\Phi(|u|/\beta)\le 1)\},
\]
$\U:=\{u\in L^1\mid \|u\|_{\Phi,r}<\infty\}$ is a Banach space with dual
\[
\U^*= \Y\oplus (L^\infty)^\perp ,
\]
where $\Y :=\{y\in L^1\mid \|y\|_{\Phi,r}^*<\infty\}$ with 
\[
\|y\|^*_{\Phi,r} = \inf_{v\in L^1}\{E[v\Phi^*(y/v)] + r^\circ(v)\}.
\]
Moreover,
\[
\|y\|_{H^*} \le \|y\|^*_{\Phi,r} \le 2\|y\|_{H^*}, 
\]
where
\begin{align*}
  \|y\|_{H^*} &= \inf\{\beta>0\mid H^*(y/\beta)\le 1\} =\inf_{v\in L^1}\max\{r^\circ(v),E[v\Phi^*(y/v)]\}.
\end{align*}

If $r$ satisfies \ref{ax5}, then the closure of $L^\infty$ in $\U$ has the expression 
\[
\tilde \U=\{u\in L^1\mid r(\Phi(|u|/\beta))<\infty\ \forall\beta>0\}.
\]
In this case, $\tilde\U=\U$ if $\dom H$ is a cone. In particular, $\dom H$ is a cone if $\Phi$ satisfies $\Delta_2$-condition: there exists $K>0$ and $x_0$ such that $\Phi(2x)\le K\Phi(x)$ for all $x\ge x_0$.
\end{example}

\begin{proof}
This fits Example~\ref{ex:lux} with
\[
H(u) :=\begin{cases}
r(\Phi(|u|))\quad&\text{if } \Phi(|u|)\in L^1,\\
+\infty\quad&\text{otherwise}.
\end{cases}
\]
For every $u\in L^1$,
\[
H(u)=\sup_{\eta\in L^\infty_+}\{E[\eta\Phi(|u|)]-r^*(\eta)\},
\]
so $H$ is lsc in $L^1$. Since $r$ satisfies \ref{ax1}--\ref{ax4}, $H$ satisfies \ref{H1}--\ref{H4}.

We compute the conjugate of $H$ by employing conjugate duality; see \cite{roc74}. Let $F(x,u):=r(\Phi(u)+x)$ be defined on $L^\infty\times L^\infty$. The conjugate $F^*$ on $L^1\times L^1$ has the expression
\begin{align*}
F^*(v,y) &:=\sup_{u,x\in L^\infty}\{E[xv+uy]-r(\Phi(u)+x)\}\\
&=\sup_{u,x\in L^\infty}\{E[vx-v\Phi(u)+uy]-r(x)\}\\
&=E[v\Phi^*(y/v)] + \delta_{B^*}(v),
\end{align*}
where the last equality comes from the interchange rule \cite[Theorem 14.60]{rw98} and
\[
B^*:=\{v\in L^1\mid r^\circ(v)\le 1\}.
\]
Since $r$ satisfies \ref{ax4}, it is $\tau(L^\infty,L^1)$-continuous?. By \cite[Theorem~17]{roc74}, this implies
\begin{align*}
H^*(y) = \inf_{v\in L^1} F^*(y,v) = \inf_{v\in L^1}\{E[v\Phi^*(y/v)]\mid r^\circ(v)\le 1\}
\end{align*}
so, by Example~\ref{ex:lux},
\begin{align*}
\|y\|_{\Phi,r}^* &= \inf_{\beta>0}\{\beta H^*(y/\beta)+\beta\}\\
&= \inf_{\beta>0,v\in L^1}\{E[\beta v\Phi^*(y/(\beta v))] + \beta\mid r^\circ(v)\le 1\}\\
&= \inf_{\beta>0,v\in L^1}\{E[v\Phi^*(y/v)] + \beta\mid r^\circ(v)\le\beta\}\\
&= \inf_{v\in L^1}\{E[v\Phi^*(y/v)] + r^\circ(v)\}.
\end{align*}
The claims concerning the dual space and its norm  follow from Example~\ref{ex:lux}. We have
\begin{align*}
  \|y\|_{H^*} &:= \inf\{\beta>0\mid H^*(y/\beta)\le 1\}\\
&= \inf\{\beta>0\mid \exists v\in L^1: r^\circ(v)\le 1,\ E[v\Phi^*(y/(\beta v))]\le 1\}\\
  &= \inf\{\beta>0\mid \exists v\in L^1: r^\circ(v)\le\beta,\ E[v\Phi^*(y/v)]\le \beta\}\\
  &=\inf_{v\in L^1}\max\{r^\circ(v),E[v\Phi^*(y/v)]\}.
\end{align*}

Assume now that $r$ satisfies \ref{ax5}. Then $H$ satisfies \ref{H5}, so Example~\ref{ex:lux} gives
\[
\tilde \U=(\dom H)^\infty.
\]
The set on the right can be written as $\{u\in L^1\mid r(\Phi(u/\beta)<\infty\ \forall\beta>0\}$.
\end{proof}

Note that if $r$ is the $L^\infty$-norm, we simply have $\U=L^\infty$ and $\Y=L^1$ while if $r$ is the $L^1$-norm, then we are back in Musielak-Orlicz spaces of Example~\ref{ex:mo}. If $\Phi$ is nonrandom and $r$ is the Lorentz-norm associated with a concave function $\phi$ (see Example~\ref{ex:ml}), $\U$ becomes the Orlicz-Lorentz-space studied e.g.\ in \cite{hkm02}. In this case the above expressions for the dual norm seem new. One could also take $r$ the Marcinkiewicz norm in which case $r^\circ$ is the Lorentz-norm. This setting seems new.

\section{On necessity of the assumptions}\label{sec:nec}

This section goes beyond Banach and Fr\'echet spaces. We assume that $\U$ and $\Y$ are solid decomposable spaces (see Remark~\ref{rem:decomp}) of random variables in separating duality under the bilinear form
\[
\langle u,y\rangle := E[u\cdot y].
\]

Clearly, solid spaces are decomposable but there are decomposable spaces that are not solid.
\begin{example}\label{ex:decomp}
Let $(\Omega,\F) :=([0,1],\B([0,1]))$, $u(\omega) :=\omega^{-\frac{1}{4}}+\omega^{-\frac{1}{2}}$ and $\U:=L^\infty +Lin (u1_A\mid A\in\F)$. Then $\U$ is decomposable, by construction,  but not solid, since it does not contain $\bar u(\omega) =\omega^{-\frac{1}{4}}$ for which $0<\bar u< u$.
\end{example}

The following two lemmas do not require solidity of $\U$ or $\Y$. The first one is Lemma~6 from \cite{pp12}. 
\begin{lemma}\label{lem:rel}
We have $L^\infty\subseteq\U\subseteq L^1$ and
\begin{align*}
\sigma(L^1,L^\infty)|_\U&\subseteq\sigma(\U,\Y),\quad\sigma(\U,\Y)|_{L^\infty}\subseteq\sigma(L^\infty,L^1),\\
\tau(L^1,L^\infty)|_\U&\subseteq\tau(\U,\Y),\quad\tau(\U,\Y)|_{L^\infty}\subseteq\tau(L^\infty,L^1).
\end{align*}
\end{lemma}

\begin{lemma}\label{lem:solid}
The following are equivalent:
\begin{enumerate}
\item
  $\U$ is solid,
\item
  $y\mapsto u\cdot y$ is continuous from $(\Y,\sigma(\Y,\U))$ to $(L^1,\sigma(L^1,L^\infty))$,
\item
  $\eta\mapsto\eta u$ is continuous from $(L^\infty,\tau(L^\infty,L^1))$ to $(\U,\tau(\U,\Y))$.
\end{enumerate}
\end{lemma}
\begin{proof}
For any $u\in\U$, $y\in\Y$ and $\eta\in L^\infty$,
\[
E[(u\cdot y)\eta] = E[(\eta u)\cdot y].
\]
This is $\sigma(\Y,\U)$-continuous in $y$ if and only if there is a $u'\in\U$ such that $E[(\eta u)\cdot y]=E[u'\cdot y]$ for all $y\in\Y$. Since $L^\infty\subset\Y$ separates the elements of $L^1$, we get that $y\mapsto E[(u\cdot y)\eta]$ is continuous if and only if $\eta u\in\U$. This proves the equivalence of 1 and 2.

Assume 2 and let $K\subset\Y$ be $\sigma(\Y,\U)$-compact. We have
\[
\sup_{y\in K}\langle y, \eta u\rangle = \sup_{y\in K}\langle u\cdot y,\eta\rangle_{L^\infty} =  \sup_{\xi\in D}\langle \xi,\eta\rangle_{L^\infty},
\]
where $D=\{u\cdot y\,|\, y\in K\}$ is $\sigma(L^1,L^\infty)$-compact since $y\mapsto u\cdot y$ is continuous.
\end{proof}

\begin{corollary}
In the setting of Corollary~\ref{cor:f2}, \ref{ax4} holds if and only if $\tilde\U^*=\Y$.
\end{corollary}

\begin{proof}
By Lemma~\ref{lem:M}, \ref{ax4} implies $\M^s=0$, so $\tilde\U^*=\Y$ by Corollary~\ref{cor:f2}. On the other hand, if $\tilde\U^*=\Y$,  the topology of $\tilde\U$ cannot be stronger than $\tau(\tilde\U,\Y)$. In that case, Lemma~\ref{lem:solid} implies that $p(u\eta^\nu)\to 0$ if $\eta^\nu\to 0$ in $\tau(L^\infty,L^1)$. Since $1_{A^\nu}\to 0$ in $\tau(L^\infty,L^1)$ if $P(A^\nu)\to 0$, assumption~\ref{ax4} holds.
\end{proof}

\begin{lemma}\label{lem:solidcomp}
A convex set $C\subset\U$ is $\sigma(\U,\Y)$-compact if and only if, for every $y\in\Y$, the set $\{u\cdot y\mid u\in C\}$ is  weakly compact in $L^1$.
\end{lemma}

\begin{proof}
Since continuous images of compact sets are compact, Lemma~\ref{lem:solid} gives the necessity. Let $(u^\nu)$ be a net in $C$. Letting $y$ range over unit constant vectors, we see that $C$ is  $\sigma(L^1,L^\infty)$-compact. Thus there is a subnet and $u\in C$ such that $u^\nu\to u$ in $\sigma (L^1,L^\infty)$. Let $y\in \Y$ and $\epsilon>0$. Since $\{u\cdot y\mid u\in C\}$ is  weakly compact in $L^1$, it is uniformly integrable, so there exists $n$ such that $|E[(u^\nu-u)\cdot y 1_{|y|> n}]|<\epsilon$ for every $\nu$. Since $u^\nu \to u$ in $\sigma(L^1,L^\infty)$, there exists $\nu'$ such that $|E[(u^\nu-u)\cdot y 1_{|y|\le n}]|<\epsilon$ for all $\nu\ge\nu'$. Thus, for all $\nu\ge\nu'$,
\[
|E[(u^\nu-u)\cdot y]| \le 2\epsilon,
\]
which proves that $u^\nu\to u$ in $\sigma(\U,\Y)$ 
\end{proof}

\begin{corollary}\label{cor:solid}
Given $\bar u\in \U$, the set 
\[
C:=\{u \in \U \mid |u| \le |\bar u|\}
\]
is $\sigma(\U,\Y)$-compact.
\end{corollary}
\begin{proof}
By Lemma~\ref{lem:solidcomp}, it suffices to show that
\[
C_y:=\{u\cdot y \mid  u\in\U, |u| \le |\bar u|\}
\]
 is $\sigma(L^1,L^\infty)$-compact for every $y\in\Y$. The set $C_y$ is uniformly integrable, so, by Dunford-Pettis, it suffices to show that $C_y$ is $\sigma(L^1,L^\infty)$-closed. Since $\U$ is solid, 
 \[
C_y=\{u\cdot y \mid  u\in L^1, |u| \le |\bar u|\}.
\]
Let $u^\nu\cdot y\to\xi$ in $L^1$, where $|u^\nu|\le |\bar u|$. Passing to convex combinations, we may assume, by Komlos lemma, that $u^\nu\to u$ almost surely for some $u$ with $|u|\le|\bar u|$. By dominated convergence, $u^\nu\cdot y\to u\cdot y$ in $L^1$, so $C_y$ is closed.
\end{proof}

\begin{theorem}\label{thm:converse}
If $\U$ is $\tau(\U,\Y)$-complete, then there exists a collection $\P$ of lsc sublinear symmetric functions $p:L^1\to\ereals$ such that the topology generated by $\P$ on $\U$ is compatible with the duality,
\[
\U=\bigcap_{p\in\P} \dom p,\quad \Y=\bigcup_{p\in\P} \dom p^\circ
\]
and each $p\in\P$ satisfies \ref{ax1}--\ref{ax5}.
\end{theorem}

\begin{proof}
Let $\C$ be the collection of $\sigma(\Y,\U)$-compact solid convex subsets of $\Y$ and let $\P$ the collection of the functions $p:L^1\to\ereals$ of the form
\[
p(u) = \sup_{y\in C}E[u\cdot y],
\]
where $C\in\C$. Each $p\in\P$ is convex and positively homogeneous. Since the unit ball of $L^\infty$ is in $\C$, \ref{ax1} holds. The topology generated by $\P$ is weaker than the Mackey-topology which is generated by all $\sigma(\Y,\U)$-compact sets. By Lemma~\ref{lem:rel}, \ref{ax2} holds. Given $\bar y\in\Y$,  $\{y\in\Y\mid |y|\le|\bar y|\}$ is compact by Corollary~\ref{cor:solid}. It is also solid and convex, so the topology generated by $\P$ is no weaker than $\sigma(\U,\Y)$. The topology generated by $\P$ is thus compatible with the duality.

Solidity of $C$ and the interchange rule~\cite[Theorem~14.60]{rw98} give
\begin{align*}
p(u) &= \sup_{y\in C,y'\in L^1}\{E[u\cdot y']\mid |y'|^*\le|y|^*\}\\
  &= \sup_{y\in C}E[|u||y|^*],
\end{align*}
so $p$ is lower semicontinuous in $L^1$ and satisfies \ref{ax3}.

By Lemma~\ref{lem:solidcomp}, the set $\{u\cdot y\mid y\in C\}$ is uniformly integrable so $p(u1_{A^\nu})\downto 0$ whenever $(A^\nu)_{\nu=1}^\infty$ is a decreasing sequence with $P(A^\nu)\downto 0$. Thus, \ref{ax5} holds. This also implies that $L^\infty$ is $\P$-dense in $\dom p$.

Any $C\in\C$ is $\sigma(L^1,L^\infty)$-compact so an application of bipolar theorem in the duality pairing $(L^1,L^\infty)$ gives
\[
p^\circ(y) = \inf\{\gamma>0\mid y/\gamma\in C\}.
\]
Thus $\dom p^\circ\subset\Y$. As noted earlier, any $y\in\Y$ belongs to some $C\in\C$ so $\Y=\cup_{p\in\P}\dom p^\circ$.

The $\sigma(\Y,\U)$-compactness of $C\in\C$ implies $\sigma_C(u)<\infty$ for any $u\in\U$. Thus, $\U\subset\cap_{p\in\P}\dom p$. On the other hand, $\cap_{p\in\P}\dom p$ is complete in the $\P$-topology (see Remark~\ref{rem:complete}) so it is complete also in the topology generated by $\sigma(\Y,U)$-compact convex sets. Since $L^\infty$ is dense in $\dom p$, we have that $\U$ is dense in $\cap_{p\in\P}\dom p$ and thus, $\U=\cap_{p\in\P}\dom p$.
\end{proof}

Theorem~\ref{thm:converse} puts us in the setting of Remark~\ref{rem:complete}. Combined with Lemma~\ref{lem:solid}, we thus get the following two results.

\begin{corollary}\label{cor:solidseq}
If $\U$ is $\tau(\U,\Y)$-complete, then it is sequentially $\sigma(\U,\Y)$-complete.
\end{corollary} 

\begin{proof}
Let $(u^\nu)_{\nu=1}^\infty$ be a $\sigma(\U,\Y)$-Cauchy sequence. Since $\sigma(\U,\Y)$ is stronger than $\sigma(L^1,L^\infty)$ which, by \cite[Theorem IV.8.6]{ds88}, is  sequentially complete, there exists $u\in L^1$ such that $u^\nu \to u$ in $\sigma(L^1,L^\infty)$.  Since $\sigma(\U,\Y)$-Cauchy sequences are bounded in any topology compatible with the pairing, the sequence is also bounded in the $\P$-topology of Theorem~\ref{thm:converse}. Thus, for any $p\in\P$, there exist $\gamma$ such that $p(u^\nu)\le\gamma$. Since level-sets of $p$ are closed in $L^1$ and $\U=\bigcap p$,  we get $u\in\U$. It suffices to show that $u^\nu \to u$ in $\sigma(\U,\Y)$. 

 By Lemma~\ref{lem:solid}, for any $y\in\Y$, $(u^\nu\cdot y)_{\nu=1}^\infty$ is Cauchy in $\sigma(L^1,L^\infty)$, so by sequential closedness of $L^1$ again, it converges in $\sigma(L^1,L^\infty)$ to some  $\xi\in L^1$. By Mazur's theorem, there is a subsequence of convex combinations $\bar u^\nu$ such that $\bar u^\nu\to u$ in $L^1$-norm, and thus  $\bar u^\nu\cdot y\to u\cdot y$ in probability. Clearly, $\bar u^\nu\cdot y \to\xi$ in $\sigma(L^1,L^\infty)$, so we must have $\xi=u\cdot y$. 
\end{proof}

When $\U$ is $\tau(\U,\Y)$-complete, we get the following version of Lemma~\ref{lem:solidcomp}.

\begin{corollary}\label{cor:solidrelcomp}
Assume that $\U$ is $\tau(\U,\Y)$-complete. A convex set $C\subset\U$ is relatively $\sigma(\U,\Y)$-compact if and only if, for every $y\in\Y$, the set $\{u\cdot y\mid u\in C\}$ is relatively $\sigma(L^1,L^\infty)$-compact in $L^1$.
\end{corollary}

\begin{proof}

Since continuous images of relatively compact sets are relatively compact, Lemma~\ref{lem:solid} gives the necessity. For the sufficiency, it suffices, by Theorem~\ref{thm:converse} and Remark~\ref{rem:complete}, to show sequential relative compactness. Let $(u^\nu)$ be a sequence in $C$. 
As in the proof of Lemma~\ref{lem:solidcomp}, we get that there is $u\in L^1$ such that, for every $y\in\Y$ and $\epsilon>0$,
\[
|E[(u^\nu-u)\cdot y]| \le 2\epsilon,
\]
for $\nu$ large enough, so $(u^\nu)$ is Cauchy in $\sigma(\U,\Y)$. By Corollary~\ref{cor:solidseq}, $(u^\nu)$ converges to $u$.
\end{proof}

\section*{Appendix}
This appendix studies integration of  measurable  not-necessarily bounded functions with respect to a real-valued finitely additive measure $m$. Define $r_m:L^1_+\to\ereals$ by
\[
r_m(\eta) := \sup_{u'\in L^\infty} \{\int_\Omega u' dm \mid |u'|\le\eta\}.
\]

\begin{lemma}\label{lem:r}
For any real-valued finitely additive measure $m$,
\begin{enumerate}
\item\label{item:rel}
Relative to $L^\infty$,
\[
r_m(\eta) = \sup_{u'\in L^\infty} \left\{ \int_\Omega \eta(u' dm) \midb |u'|\le 1\right\} \le ||\eta||_{L^\infty}||m||_{TV}.
\]
In particular, $r_m$ is $L^\infty$-norm continuous and sublinear relative to $L^\infty_+$.
\item\label{item:proj}
For every $\eta\in L^1_+$,
\[
r_m(\eta)=\lim_{\nu\nearrow\infty} r_m(\eta\wedge \nu)
\]
\item\label{item:phsa}
$r_m$ is positively homogeneous and subadditive and $r_m(\eta')\le r_m(\eta)$ whenever $\eta'\le \eta$.
\end{enumerate}
\end{lemma}

\begin{proof}
The expression in \ref{item:rel} follows from the change of variables $\tilde u= \eta u'$. To prove \ref{item:proj}, the inequality $r_m(\eta)\ge \lim_\nu r_m(\eta\wedge \nu)$ is clear. To prove the opposite inequality, let $\alpha\in\reals$ with $r_m(\eta)>\alpha$. There exists $u'\in L^\infty$ with $|u'|\le \eta$ and $r_m(|u'|)>\alpha$. Then $|u'|\wedge \nu \to |u'|$ in $L^\infty$-norm, so monotonicity and \ref{item:rel} give
\[\lim r_m(\eta\wedge \nu)\ge \lim r_m(|u'|\wedge  \nu ) >\alpha.
\] 

In \ref{item:phsa}, only subadditivity requires a proof. Given $\eta^1,\eta^2\in\dom p$, we have $(\eta^1+\eta^2)\wedge \nu\le\eta^1\wedge \nu+\eta^2\wedge \nu$. Indeed, a concave function vanishing at the origin is subadditive on the positive reals. Thus, by \ref{item:rel} and \ref{item:proj},
\begin{align*}
r_m(\eta^1+\eta^2)&=\limsup_\nu r_m((\eta^1+\eta^2)\wedge \nu)\\
&\le\limsup_\nu(r_m(\eta^1\wedge \nu)+r_m(\eta^2\wedge \nu))\\
&\le\limsup_\nu r_m(\eta^1\wedge \nu)+\limsup_\nu r_m(\eta^2\wedge \nu)\\
&=r_m(\eta^1)+r_m(\eta^2),
\end{align*}
which proves the subadditivity.
\end{proof}

Define $\rho_m:L^1\to\ereals$ by
\[
\rho_m(u) := r_m(|u|).
\]

\begin{theorem}\label{thm:ext1d}
For any real-valued finitely additive measure $m$,
\begin{enumerate}
\item\label{item:rhophsa}
$\rho_m$ is symmetric and sublinear, and $\rho_m(u')\le \rho_m(u)$ whenever $|u'|\le|u|$,
\item\label{item:dense}
for any $u\in\dom\rho_m$ and $\epsilon>0$, there exists $u'\in L^\infty$ with $\rho_m(u-u')<\epsilon$,
\item $\int_\Omega u dm$ has a unique $\rho_m$-continuous linear extension from $L^\infty$ to $\dom\rho_m$,
\item if $m$ is purely finite additive, there exists a decreasing $(A^\nu)_{\nu=1}^\infty\subset\F$ with $P(A^\nu)\searrow 0$ and $\int_\Omega u1_{\Omega\backslash A^\nu} dm =0$ for all $u\in\dom\rho_m$.
\end{enumerate}
\end{theorem}

\begin{proof}
Properties in 1 are clear. To prove \ref{item:dense}, assume first that $m$ is nonnegative.  Given $u^i\in\dom\rho_m\cap L^1_+$ and $\epsilon>0$, let $\tilde u^i \in L^\infty$ be such that $0\le \tilde u^i\le u^i$ and $\rho_j(u^i)\le\langle\tilde u^i,m\rangle+\epsilon$. Then $\tilde u^1+\tilde u^2\le u^1+u^2$ and
\[
\rho_m(u^1)+\rho_m(u^2) \le \langle \tilde u^1+\tilde u^2,m\rangle + 2\epsilon \le \rho_m(u^1+u^2) + 2\epsilon.
\]
Since $\epsilon>0$ was arbitrary, $\rho_m$ is superlinear on $\dom\rho_m\cap L^1_+$. Given $u\in\dom\rho_m$ and $\epsilon>0$, Lemma~\ref{lem:r} gives $\rho_m(u^+)\le \rho_m(u^+\wedge \nu)+\epsilon$ for $\nu$ large enough. By superlinearity, 
\[
\rho_m(u^+-u^+\wedge\nu)+\rho_m(u^+\wedge\nu)\le \rho_m(u^+) \le \rho_m(u^+\wedge\nu)+\epsilon.
\]
Similarly, $\rho_m(u^-- u^-\wedge \nu)\le \epsilon$, so $\rho_m(u-\pi_{\nu\uball} u)\le 2\epsilon$ by sublinearity of $\rho_m$. By \cite[Theorem 1.12]{yh52}, general $m\in\M^1$ can be written as $m=m^+-m^-$ for nonnegative $m^+,m^-\in\M^1$, so
\[
\rho_m(u-\pi_{\nu\uball} u) \le \rho_{m^+}(u-\pi_{\nu\uball} u)+\rho_{m^-}(u-\pi_{\nu\uball} u) \le 4\epsilon
\]
for $\nu$ large enough.

We have $\int_\Omega u dm\le\rho_m(u)$ on $L^\infty$, so, by Hahn-Banach, there exists a $\rho_m$-continuous linear extension of $m$ to $\dom\rho_m$. Since $L^\infty$ is dense in $\dom\rho_m$, the extension is unique. If $m$ is purely finitely additive, there exists $(A^\nu)_{\nu=1}^\infty\subset\F$ with $P(A^\nu)\searrow 0$ and $\int_\Omega u1_{\Omega\backslash A^\nu} dm =0$ for all $u\in L^\infty$. Note that $r_m$ inherits this property so that $\rho_m$ and the integral does as well.
\end{proof}

\bibliographystyle{plain}
\bibliography{sp}

\end{document}